\documentclass[12pt, reqno]{amsart}
\usepackage{amsmath, amsthm, amscd, amsfonts, amssymb, graphicx, color}
\usepackage[bookmarksnumbered, colorlinks, plainpages]{hyperref}
\hypersetup{colorlinks=true,linkcolor=red, anchorcolor=green, citecolor=cyan, urlcolor=red, filecolor=magenta, pdftoolbar=true}

\newtheorem{theorem}{Theorem}[section]

\newtheorem{proposition}[theorem]{Proposition}
\newtheorem{corollary}[theorem]{Corollary}
\theoremstyle{definition}
\newtheorem{definition}[theorem]{Definition}
\newtheorem{example}[theorem]{Example}

\newtheorem{problem}[theorem]{Problem}
\theoremstyle{remark}
\newtheorem{remark}[theorem]{Remark}
\numberwithin{equation}{section}

\begin{document}

\setcounter{page}{1}

\title[Quasi hyperrigidity and weak peak points]{Quasi hyperrigidity and weak peak points for non-commutative operator systems}

\author[M. N. N. Namboodiri,  S.PRAMOD, P.SHANKAR \MakeLowercase{and} A.K. VIJAYARAJAN]{M. N. N. Namboodiri$^{1}$, S. PRAMOD$^2$, P. SHANKAR$^3$ \MakeLowercase{and} A.K. VIJAYARAJAN$^4$}

\address{$^1$ Department of  Mathematics, Cochin University of Science \& Technology, Kochi, Kerala, India-682022}
\email{mnnadri@gmail.com}

\address{$^{2}$ Department of Agriculture, Government of Kerala, Kerala, India.}
\email{pramodsuryan@yahoo.com}

\address{$^{3}$ Kerala School of Mathematics, Kozhikode, Kerala, India-673571.}
\email{shankarsupy@gmail.com, shankar@ksom.res.in}

\address{$^{4}$ Kerala School of Mathematics, Kozhikode, Kerala, India-673571.}
\email{vijay@ksom.res.in}



\date{\today}

\begin{abstract} 

In this article, we introduce the notions of weak boundary representation, quasi hyperrigidity and weak peak points in the non-commutative setting  for operator systems in $C^*$-algebras. An analogue of Saskin's theorem relating quasi hyperrigidity and weak Choquet boundary for particular classes of $C^*$-algebras is proved. We also show that, if an irreducible representation is a weak boundary representation and weak peak then it is a boundary representation. Several examples are provided to illustrate these notions. It is also observed that isometries on Hilbert spaces play an important role in the study of certain operator systems.
\end{abstract} 

\maketitle

\section{Introduction}
The concepts of peak point and Choquet boundary play an important role in several areas of classical (commutative) analysis. To be more precise the idea was to identify optimal subsets of a compact Hausdorff space $X$ such that each and every element of a given class of continuous complex functions $C(X)$ attains maximum modulus on it. These notions are also related to classical approximation theory studied extensively with the concept of Korovkin (hyperrigid sets) sets. The idea of peak points was introduced by Bishop in connection with the study of Choquet boundary of subspaces of $C(X)$ and developed further by Saskin \cite{Sas} in the geometric formulation of the so called `Korovkin' sets. The non-commutative counterparts of these notions for operator systems and operator algebras were initiated by William Arveson in \cite{Arv0} and studied extensively by him in \cite{Arv1}, \cite{Arv2} and \cite{Arv3}. In \cite{Arv1} and \cite{Arv2} Arveson introduced the concept of non-commutative Choquet boundary and proved a number of results related to hyperrigid sets which are the non-commutative analogues of classical Korovkin sets in approximation theory of functions. A brief survey of the developments in `non-commutative Korovkin-type theory'  is given in \cite{Nam1} and a non-commutative version of the geometric theory of Korovkin sets  can be found in \cite{Nam2}.

In this article, we introduce the notion of a {\it weak peak point} for an operator system in a $C^*$-algebra which is a non-commutative analogue of peak point. This is with an intention to identify `distinguished' points that are `peak' in an appropriate sense that involves only simple machinery compared to those of peaking representation introduced by Arveson \cite{Arv2}. We introduce {\it quasi hyperrigid sets} in $C^*$-algebras which are weaker than hyperrigid sets. We also introduce {\it weak boundary representations} and study the relation between boundary representations and weak boundary representations for operator systems of $C^*$-algebras. We prove an analogue of Saskin's theorem relating quasi hyperrigid operator systems and weak boundary representations for operator systems of $C^*$-algebras.

 The main result of the paper is as follows.\\
{\bf Theorem:}~~Let $S$ be an operator system in a $C^*$-algebra $A = C^*(S)$. If $\pi\in \hat{A}$ is a weak peak point for $S$, $\pi$ is a weak boundary representation for $S$ and $\pi_{|_S}$ is pure, then $\pi$ is a boundary representation for $S$.

\section{Preliminaries}

Now we recall the necessary definitions in the commutative case which we are going to adapt appropriately to the non-commutative setting.

\begin{definition}
Let $X$ be a compact Hausdorff space and let $C(X)$ be the set of continuous complex valued functions on $X$. The subset $S$ of $C(X)$ is called a {\it Korovkin set} if whenever $\{\phi_n\}$ is a sequence of positive linear maps from $C(X)$ to itself such that $||\phi_n(f)-f||\rightarrow 0$ for all $f\in S$, then $||\phi_n(f)-f||\rightarrow 0$ for all $f\in C(X)$.
\end{definition}

The classical Korovkin theorem proves that $\{1,x,x^2\}$ is a Korovkin set for $C([0,1])$.

\begin{definition}
Let $X$ be a compact Hausdorff space and $G$ a closed subspace of $C(X)$, separating points and containing the identity of $C(X)$. A point $x_0 \in X$ is a {\it peak point} of $G$ if there exists a $g \in G$ for which $g(x_0) = \|g\|,\linebreak |g(x)| <\|g\|,~ x \neq x_0$.
\end{definition}

\begin{definition}
Let $S\subset C(X)$ containing the constant function $1$, where $X$ is a compact Hausdorff space. The Choquet boundary $\partial{S}$ of $S$ is defined as 
$\partial{S}=\{x\in X : \varepsilon_{x|_S} \text{ has a unique positive linear extension to}~ C(X), \text{where}~ \varepsilon_x ~ \text{denotes}\linebreak \hspace*{2cm}\text{the evaluation functional defined by}~ \varepsilon_x (f)=f(x),f\in C(X) \}. $
\end{definition}

The above two notions are closely related. One of the most significant results connecting them is as follows:
 
\begin{theorem}(\cite{BL}, page-170)
Let $X$ be a compact Hausdorff space and let G be the linear space spanned by a subset $S$ of $C(X)$ that contains constants and separates points of $X$. Then the set of peak points of $G$ is contained in the Choquet boundary of $G$.
\end{theorem}

Here we give the most remarkable and well celebrated theorem proved by Saskin in \cite{Sas}
\begin{theorem}\label{sast}
Let $S$ be a subset of $C(X)$ that separates points of $X$ and contains constant function. Then $S$ is a Korovkin set in $C(X)$ if and only if the Choquet boundary $\partial S = X$. 
\end{theorem}

Our main aim of this paper is proving the non commutative analogues of the above theorems. To fix our notation and terminology we recall the fundamental notions. 

The {\it spectrum} $\hat{A}$ of a $C^*$-algebra ${A}$ is the set of all unitary equivalence classes of irreducible representations of ${A}$ on a Hilbert space.
An {\it operator system} $S$ in a $C^*$-algebra ${A}$ is a self-adjoint linear subspace of ${A}$  containing the identity of ${ A}$ such that ${ A} = C^*(S)$-the $C^*$-algebra generated by $S$. If $S \subseteq B(H)$-the space of all bounded linear operators on some Hilbert space $H$, then it is called a concrete operator system.
We use the notation $CP({A}, H)$ to denote the set of all completely positive (CP) maps from the $C^*$-algebra ${A }$ to  $B{\left(H\right)}$. By $UCP({A}, {H})$ we denote the subset of completely positive maps that are unital (UCP). A map $\phi \in UCP({A}, {H})$ is called $\it pure$, if whenever $\phi - \xi$ is completely positive for some $\xi  \in CP({A}, {H})$, there exists $0 \leq t \leq1$ such that $\xi = t \phi$. When $H$ is finite dimensional, elements of $ UCP({A}, {H})$ are called  $\it matrix~ states$. Let $UCP(A,\pi,H_{\pi})=\{\Phi: \Phi(\cdot)=V^*\pi(\cdot)V,~ V ~\text{is an isometry on} ~H_{\pi}\}$.

The following important notion of boundary representations in the theory of operator systems was introduced by Arveson \cite {Arv0}:

\begin{definition}
Let $S$ be an operator system in a $C^*$-algebra ${A}$.  A {\it boundary representation} for $S$ is an irreducible representation $\pi$ of $A$ on a Hilbert space such that $\pi_{\mid_ S}$ has a unique completely positive extension, namely $\pi$ itself to $A$. 
\end{definition}

 The set of all boundary representations for $S$ is called the {\it non-commutative Choquet boundary} of $S$ and is denoted by $\partial(S)$. 

The non-commutative approximation theory initiated by Arveson benefited remarkably from the theory of boundary representations. In \cite{Arv2} Arveson coined the term `hyperrigid set' when he introduced the non-commutative version of Korovkin set in commutative set-up and it is as follows. 

\begin{definition}
A set $G$ of generators of an abstract $C^*$-algebra $A$ is said to be {\it hyperrigid} if for every faithful representation $A \subseteq B(H)$ of $A$ on a Hilbert space and every sequence of unital completely positive maps $\{\phi_n\}$ from $B(H)$ to itself,
$$\lim\limits_{n\to\infty} \|\phi_{n}(g)-g\| = 0, ~~~\forall~ g \in G \Rightarrow \lim\limits_{n\to\infty} \|\phi_{n}(a)-a\| = 0, ~~~\forall~ a \in A.$$
\end{definition}

The non commutative analogue of peak point introduced by Arveson in \cite{Arv2} is as follows.

\begin{definition}
Let $S$ be a separable operator system and let $A=C^*(S)$ is the $C^*$-algebra generated by $S$. An irreducible representation $\pi:A \rightarrow B(H)$ is said to be a {\it peaking representation} for $S$ if there is an $n\geq 1$ and an $n\times n$ matrix $(s_{ij})$ over $S$ such that 
$$\mid\mid (\pi(s_{ij}))  \mid\mid > \mid\mid (\sigma(s_{ij}))  \mid\mid  $$
for every irreducible representation $\sigma$ not unitarily equivalent to $\pi$.
\end{definition}

\section{Weak choquet boundary and Quasi hyperrigid sets}

Here we introduce the notion of a weak boundary representation and discuss the nature and properties  of it. 

\begin{definition}
Let $A$ be a unital $C^*$-algebra and $S$ be an operator system of $A$ such that $A=C^*(S)$-the $C^*$-algebra generated by $S$. An irreducible representation $\pi:A\rightarrow B(H_{\pi})$ is called {\it weak boundary representation} for $S$ of $A$ if $\pi_{|_S}$ has a unique UCP map extension of the form $V^*\pi V$, namely $\pi$ itself, where $V:H_{\pi}\rightarrow H_{\pi}$ is an isometry.
\end{definition}

The set of all weak boundary representations for $S$ of $A$ is called weak Choquet boundary of $S$ and denoted by $\partial_W S$. We can observe that all the boundary representations are weak boundary representations for $S$. Thus $\partial S\subseteq \partial_W S$.

\begin{example}
Consider the classical case $A=C(X)$, where $X$ is a compact Hausdorff space. The irreducible representations up to unitary 
equivalence are one dimensional representations of $C(X)$ which correspond to point evaluation functionals and thereby precisely to the points of $X$. Let $S$ be a subspace of $C(X)$ containing identity such that $C^*(S)=C(X)$. Let $x\in X$, $\varepsilon_x:C(X)\rightarrow \mathbb{C}$ be the one dimensional irreducible representation given by $\varepsilon_x(f)=f(x),~~\text{for all}~~ f\in C(X)$. Let $V:\mathbb{C}\rightarrow \mathbb{C}$ be an isometry such that $V^*\varepsilon_x(f)V=\varepsilon_x(f)$ for all $f\in S$. Since $\mathbb{C}$ is one dimensional, $V$ is unitary and hence $V^*\varepsilon_x(f)V=\varepsilon_x(f)$ for all $f\in C(X)$. Therefore $\varepsilon_x$ is a weak boundary representation for all $x\in X$. In the classical case, spectrum of a $C^*$-algebra and weak Choquet boundary are the same irrespective of the choice of   the subspace $S$ of $C(X)$. Hence we conclude that $\partial S\subseteq \partial_W S =X$. By Saskin's theorem \ref{sast}, we conclude that a subspace $S$ is Korovkin in $C(X)$ if and only if $\partial S=\partial_W S=X$. Thus weak Choquet boundary fails to recognise hyperrigidity even in the commutative case since $\partial_W S =X$ for all $S\subseteq C(X)$.
\end{example}

\begin{example}\label{weakb}
Let $A$ be a $C^*$-algebra and $S$ be a operator system in $A$ such that $A=C^*(S)$, when $A$ is finite dimensional it is easy to see that $\partial_W S =\hat{A}$. The same can be deduced for infinite dimensional $C^*$-algebras for which all the irreducible representations are finite dimensional as in the cases of infinite direct sum of matrix algebras and infinite direct sums of the form  $ \oplus{( C(X_i)\otimes{M_n(\mathbb C)}})$, where $ X_i$ is a compact Hausdorff space for each $i$.
\end{example}

The notion of weak boundary representation is interesting in the infinite dimensional $C^*$-algebras. The following example shows that spectrum of a $C^*$-algebra is not always equal to weak Choquet boundary in infinite dimensional cases.

\begin{example}
Let $G$ = linear span$\left(I,S,S^*\right)$, where $S$ is the unilateral right shift in $B(H)$  and $I$ is the identity operator. Let $A=C^{*}(G)$ be the $C^*$-algebra generated by $G$.  We have  $K(H) \subseteq A$, $A/K(H)\cong {C}(\mathbb{T})$ is commutative, where $\mathbb{T}$ denotes the unit circle in $\mathbb{C}$ and the spectrum $\hat{A}$ of $A$ can be identified with $\{Id\}\cup \mathbb{T}$. We know that $\varepsilon_t$ is a one dimensional irreducible representation of $A$ for all $t\in \mathbb{T}$, therefore $\varepsilon_t$ is a weak boundary representation for $G$ of $A$ for all $t\in \mathbb{T}$. Note that $Id_{|_G}$ has more than one UCP extension from the class $CP(A,Id,H_{Id})$. Observe that $S^*Id(\cdot)S$ is also an extension of $Id_{|_G}$ therefore $Id$ is not a weak boundary representation.
\end{example}

  Here we introduce the notion of quasi hyperrigid sets and discuss the relation between  quasi hyperrigidity and other notions.
  
\begin{definition}
        A  set $S$ of generators of a $C^*$-algebra $A$ is said to be quasi hyperrigid,  if for every nondegenerate representation $\pi$ of $A$ on a Hilbert space $H_{\pi}$ and for every isometry $V:H_{\pi}\rightarrow{H_{\pi}}$ the condition $ V^*\pi(s) {V}=\pi(s)$ for all $s$ in $S$ implies that $V^*\pi(a)V =\pi(a)$ for all $a$ in $A$.
\end{definition}

Note that a set $S$ is quasi hyperrigid if and only if the linear span of $S \cup S^*$ is quasi hyperrigid and hence the notion extends naturally to operator systems.\\

Here we explore the relation between hyperrigidity and quasi hyperrigidity. It is trivial to see that hyperrigid sets are quasi hyperrigid. However, the converse is not true and hence the notion is strictly weaker. We illustrate this using several examples. The following one is a modified version from \cite{LN}. 

\begin{example}
 Let $ {M}_{n}(\mathbb C)$ denote the set of all $n\times n$ matrices over $\mathbb C$, where $ n\geq 3$. Define a unital completely positive map $\Phi$ on  $  M_{n}(\mathbb{C})$ as given below.
             Let             
          \[   M=\left[\begin{array}{ccccc}
             a_{11} & a_{12} & a_{13} &...... & a_{1n} \\ 
             a_{21} & a_{22} & a_{23} & ......& a_{2n} \\
             a_{31} & a_{32} & a_{33} & ......& a_{3n} \\
             .& .& .&.... ..& . \\
             .& .& .&.... ..& . \\
             a_{n1} & a_{n2} & a_{n3} & ......& a_{nn}
            
            \end{array}
            \right] \]
            be arbitrary. Now define $\Phi$ on ${M}_{n}\mathcal{(\mathbb{C})}$
            
              \[
            \Phi(M)=\left[\begin{array}{ccccc}  
             a_{11} & a_{12} & 0 &...... &0\\ 
             a_{21} & a_{22} & 0 & ......& 0\\
             0& 0 & a_{22} & ......& 0\\
             0&0 &0 & a_{22}&0 \\
             .& .& .&.... ..& . \\
            0 & 0 & 0& ......& a_{22}
                        \end{array}\right]
                        \]
            Now let $M=T$, where $a_{21}=1$ and all other entries equal to $0$. If $S = span\{I,T,T^*\}$ and $A = C^*(S)$, then $\Phi(s)= s $ for all $s$ in $S$, but 
            $\Phi(T{T}^*) \neq T{T}^*$. i.e, $S$ is not a hyperrigid set. However, if $V$ is any isometry such that $V^*V=I$, then $VV^*= I$, since $A$ is finite dimensional. Thus $S$ is quasi hyperrigid, but fails  to be a hyperrigid set. 
\end{example}

Now we give an infinite dimensional example of a quasi hyperrigid operator system which is not hyperrigid. This example is inspired by Robertson \cite{AGR}. In fact a slight modification of Robertson's construction of the CP map is made so as to make it unital. We choose an operator system in such a way that the example fit into our settings.

\begin{example}
We assume the construction of the theorem (\cite{AGR}, page 472). Let $A$ be a non commutative infinite dimensional $C^*$-algebra with only finite dimensional irreducible representations and define a modified version of the CP map $\phi$ on $A$ as given in equation (1) in \cite{AGR} as follows. Start with an $x$ in $A$ with spec($x$) $ =0$. Now define 
$$ \phi(a)=pap+\sigma(a)(I-p^2) $$
for each $a$ in $A$, where $p$ is the projection as described by Robertson. Then proceed as in \cite{AGR} to construct the required example by considering different cases for spec$(x^*x)$.
\end{example}

\begin{remark}
In the infinite dimensional $C^*$-algebras considered in example \ref{weakb} we can construct quasi hyperrigid operator systems which are not hyperrigid.
\end{remark}

Now we explore the notions of quasi hyperrigidity and weak Choquet boundary in the following results.

\begin{proposition}\label{unique}
Let $S$ be a separable operator system $S$ and $A=C^*(S)$. Then $S$ is quasi hyperrigid if and only if for every nondegenerate representation $\pi:A\rightarrow B(H_\pi)$ on a separable Hilbert space, $\pi_{|_S}$ has a unique UCP map extension of the form $V^*\pi V$, where $V:H_\pi \rightarrow H_\pi$ is an isometry.
\end{proposition}
\begin{proof}
 The result is immediate from the definition of quasi hyperrigidity.
\end{proof}

\begin{proposition}
Let $S$ be a separable operator system generating a $C^*$-algebra $A$. If $S$ is quasi hyperrigid, then every irreducible representation of $A$ is a weak boundary representation for $S$. 
\end{proposition}
\begin{proof}
The assertion is an immediate consequence of above proposition.
\end{proof}

\begin{problem}\label{prblm}
If every irreducible representation of $A$ is a weak boundary representation for a separable operator system $S\subseteq A$, then is $S$  quasi hyperrigid?
\end{problem}

We will settle the above problem for certain classes of $C^*$-algebras.

\begin{proposition}\label{exten}
Let $S$ be a operator system generating a $C^*$-algebra $A=C^*(S)$ and for each $i$ in an index set $I$, let $\pi_i : A \rightarrow B(H_{\pi_i})$ be a representation such that $\pi_{{i|}_S}$ has unique UCP map extension of the form $V_{\pi_i}^*\pi_i V_{\pi_i}$, where $V_{\pi_i} : H_{\pi_i} \rightarrow H_{\pi_i}$ is an isometry. Then for the direct sum of representations $\pi=\oplus_{i\in I} \pi_i : A \rightarrow B(\oplus_{i \in I} H_{\pi_i})$, $\pi_{|_S}$ has unique UCP map extension of the form $V^*_{\pi}\pi V_{\pi}$, where $V_{\pi}:\oplus_{i \in I} H_{\pi_i} \rightarrow \oplus_{i \in I} H_{\pi_i}$ is an isometry.
\end{proposition}

\begin{proof}
Let $\Phi=V^*_{\pi}\pi V_{\pi}=V^*_{\pi}\oplus_{i\in I} \pi_i  V_{\pi} : A \rightarrow B(\oplus_{i \in I} H_{\pi_i})$ be an extension of $\pi_{|_S}$ for an isometry $V_{\pi}:\oplus_{i \in I} H_{\pi_i} \rightarrow \oplus_{i \in I} H_{\pi_i}$. For each $i \in I$, let $\Phi_i: A \rightarrow B(H_{\pi_i})$ be the UCP map 
$$\Phi_i(a)=P_i\Phi(a)|_{H_{\pi_i}},~~~~~~~~ a \in A $$
where $P_i$ is the projection onto $H_{\pi_i}$. Since $\Phi_i$ restricted to $\pi_i$ on $S$  has unique extension we have $\Phi_i(a)=\pi_i(a)$ for all $a\in A$. Equivalently $P_i\Phi(a)P_i=\pi(a)P_i$. Using Schwarz inequality,
$$\begin{array}{rcl}
P_i\Phi(a)^*(1-P_i)\Phi(a)P_i& = &P_i\Phi(a)^*\Phi(a)P_i - P_i\Phi(a)^* P_i \Phi(a)P_i\\
                             &\leq & P_i\Phi(a^*a)P_i- \pi(a)^* P_i \Phi(a)P_i \\
                             & = & \pi(a^*a)P_i -\pi(a)^*\pi(a)P_i \\
                             & = & 0
\end{array}$$
Hence $\left|(1-P_i)\Phi(a)P_i \right|^{2}=0$, and it follows that $P_i$ commutes with the self adjoint family of operators $\Phi(A)$. Hence for every $a\in A$ we have
$$ \Phi(a)=\sum_{i\in I} \Phi(a)P_i=\sum_{i\in I} P_i\Phi(a)P_i = \sum_{i\in I} \pi(a)P_i = \pi(a) $$ 

Hence $\Phi(a)=V^*_{\pi}\pi(a) V_{\pi}=V^*_{\pi}\oplus_{i\in I} \pi_i(a)  V_{\pi} = \pi(a)$ for all $a\in A$ and for an isometry  $V_{\pi}:\oplus_{i \in I} H_{\pi_i} \rightarrow \oplus_{i \in I} H_{\pi_i}$.

\end{proof}

Now we settle the problem \ref{prblm} for $C^*$-algebras with countable spectrum.

\begin{theorem}
Let $A=C^*(S)$ be the $C^*$-algebra generated by a separable operator system $S$ such that $A$ has countable spectrum. If every irreducible representation of $A$ is a weak boundary representation for $S$ then $S$ is quasi hyperrigid. 
\end{theorem} 

\begin{proof}
To prove $S$ is quasi hyperrigid using proposition \ref{unique}, it is enough to prove that for every representation $\pi:A\rightarrow B(H_\pi)$ of $A$ on a separable Hilbert space, $\pi_{|_S}$ has the unique UCP map extension of the form $V_\pi^*\pi V_\pi$, where $V_\pi:H_\pi \rightarrow H_\pi$ is an isometry. Our assumption that spectrum of $A$ is countable implies that $A$ is a type I $C^*$-algebra, hence $\pi$ decomposes uniquely into a direct integral of mutually disjoint type I factor representations. Using the fact that spectrum of $A$ is countable again, the direct integral must be a countable direct sum. therefore $\pi$ can be decomposed into a direct sum of subrepresentations $\pi_n:A\rightarrow B(H_{\pi_n})$
$$H_\pi=H_{\pi_1}\oplus H_{\pi_2}\oplus...,~~~~~\pi=\pi_1\oplus \pi_2\oplus... $$  
with the property that each $\pi_n$ is unitarily equivalent to a finite or countable direct sum of copies of a single irreducible representation $\sigma_n:A\rightarrow B(H_{\sigma_n})$.

By our assumption, each map $\sigma_{{n|}_S}$ has the unique UCP map extension of the form $V_{\sigma_n}^*\sigma_n V_{\sigma_n}$, where $V_{\sigma_n}:H_{\sigma_n} \rightarrow H_{\sigma_n}$ is an isometry. Hence the above decomposition expresses $\pi_{|_S}$ as a (double) direct sum. By Proposition \ref{exten} it follows that $\pi_{|_S}$ has the unique UCP map extension of the form $V_\pi^*\pi V_\pi$, where $V:H_\pi\rightarrow H_\pi$ is an isometry.
\end{proof}

\section{Unique extensions-weaker versions}

In this section, we introduce the weaker notion of unique extension property of representations of $C^*$-algebras by considering particular class of UCP maps. 

\begin{definition}
Let $S$ be a operator system generating a $C^*$-algebra $A$. Let $\pi: A \rightarrow B(H_\pi)$ be a representation then $\pi$ is said to have weak unique extension property(WUEP) for $S$ if $\pi$ is the only UCP map extension of $\pi_{|_{S}}$ of the form $V^*\pi(\cdot)V$, where $V$ is an isometry on $H_{\pi}$.
\end{definition}

Kleski \cite{Klek} proved the hyperrigid conjecture of Arveson for a Type I $C^*$-algebras with an additional assumption on the codomain. Since our problem \ref{prblm} is similar to Arveson's hyperrigid conjecture with weaker notions, Kleski's \cite{Klek} results can be modified to our settings. The following results give partial answer to the problem \ref{prblm}. Since the arguments are exactly the same verbatim, we will state results without proof.

\begin{theorem}
Let $S$ be a separable operator system in $B(H)$ generating a $C^*$-algebra $A$, and suppose $A''$ is injective. 
Suppose every factor representation $\pi:A\rightarrow B(H_\pi)$ has WUEP for $S$ of $A$. Let $\rho$ be a faithful representation of $A$ on $B(K_\rho)$ and let $\gamma:\rho(A)\rightarrow B(K_\rho)$, $\gamma=V_1^*Id V_1$, where $V_1:K_\rho\rightarrow K_\rho$ is an isometry such that $\gamma(\rho(s))=\rho(s)$ for all $s\in S$. Then for every conditional expectation $E:B(K_\rho)\rightarrow \rho(A)''$, we have $E\gamma \rho (a)=\rho(a)$ for all $a\in A$.
\end{theorem}
 
\begin{corollary}
Let $S$ be an operator system generating a Type $I~C^*$-algebra $A$. If every irreducible representation of $A$ is a weak boundary representation for $S$, then for any representation $\pi:A\rightarrow B(K_\pi)$ and any UCP map $V^*IdV:\pi(A)\rightarrow B(K_\pi)$ for $V:K_\pi\rightarrow K_\pi$ is isometry such that $V^*Id(\pi(s))V=\pi(s)$ for all $s\in S$ and any conditional expectation $E:B(K_\pi)\rightarrow \pi(A)''$, $E(V^*IdV)\pi=\pi$.
\end{corollary}

\begin{corollary}
Let $S$ be a separable operator system generating a Type $I~C^*$-algebra $A$. If every irreducible representation of $A$ is a weak boundary representation for $S$, then for any UCP map $V^*\pi V: A\rightarrow A''$, where $\pi:A\rightarrow A''$ is a representation and $V\in A''$ is an isometry such that $V^*\pi(s)V=\pi(s)$ for all $s\in S$ implies that $V^*\pi(a)V=\pi(a)$ for all $a\in A$.
\end{corollary}

\section{Weak peak points}
In this section we will introduce the notion of weak peak point which is a non-commutative analogue of peak point but different from Arveson's peaking representation.

\begin{definition}
Let $A$ be a unital $C^*$-algebra and $S$ be an operator system of $A$ such that $A = C^*(S)$, the $C^*$-algebra generated by $S$. An element $\pi$ of $\hat{A}$ is called a {\it weak peak point} for $S$ if there exists $s \in S$ such that
\begin{itemize}
\item [(i)] $|\langle \pi(s) \xi_\pi , \xi_\pi  \rangle| =\|s\|$  for some $\xi _\pi \in H_\pi$ with  $\|\xi_{\pi}\| = 1$,
\item [(ii)] $|\langle \sigma(s) \xi_{\sigma}  , \xi_{\sigma}  \rangle|< \|s\|$ for all $\xi _\sigma \in H_\sigma$ with $\|\xi_{\sigma}\| = 1$,
\end{itemize}
where $\sigma$ is any  irreducible representation not equivalent to $\pi$.
 We will denote the set of all weak peak points for $S$ by $P_w(S)$.
\end{definition}

However the exact relation between weak peak points and peaking representations of an operator system calls for further study.\\

We observed that the Choquet boundary of an operator system is contained in weak Choquet boundary of it and this inclusion is strict. So it would be interesting to know which weak Choquet boundary points are Choquet boundary points of an operator system. The following theorem gives partial answer to this query.

\begin{theorem}\label{thm1}
Let $S$ be an operator system in a $C^*$-algebra $A = C^*(S)$. If $\pi\in \hat{A}$ is a weak peak point for $S$, $\pi$ is a weak boundary representation for $S$ 
and $\pi_{|_S}$ is pure, then $\pi$ is a boundary representation for $S$.
\end{theorem}
\begin{proof}
Let $\pi \in P_w(S)$. We want to show that $\pi$ is a boundary representation for $S$.\linebreak Let $K= \left\{ \Psi \in CP\left({A}, {H}_{\pi}\right): \Psi_{\mid_S} = {\pi}_{\mid_S} \right\}$. Then $K$ is a compact convex set with respect
to the BW-topology (\cite{Arv0}, page 146). By Krein-Milman theorem, there exists an extreme element $\Phi$ of $K$. Since $\Phi$ is linearly extreme and ${\Phi}_{\mid_S}$ is pure, $\Phi$ is pure (\cite {Kle}, Proposition 2.2 and Corollary 2.3). Let $(V, {H}_{\pi'}, \pi' )$ be the minimal Stinespring triple corresponding to $\Phi$ where $\pi'$ is an irreducible representation. Then $\Phi(.)=V^*\pi'(.)V$. Since $\Phi$  is unital, $\Phi(1_A) = V^*\pi'(1_A )V = V^*V= I$, so $V$ is isometric.
 
Now we show that $\pi' \sim \pi$. Let if possible, $\pi$ is not equivalent to $ \pi'$. Since $\pi \in P_{w}(S)$, there exists $s\in S$ such that
 $$
 |\left\langle \pi(s)\xi_{\pi} , \xi_{\pi}\right\rangle| =  \left\|s\right\| \mbox{for some unit vector}~ \xi_{\pi},~~~  \mbox {and}$$  
$$|\left\langle \pi'(s)\xi_{\pi'} , \xi_{\pi'}\right\rangle| < \left\|s\right\| \mbox{for all unit vectors}~ \xi_{\pi'}. 
 $$
  Now,
$$
\|s\| = |\left<\pi(s)\xi_{\pi} , \xi_{\pi}\right>| = |\left< \Phi(s)\xi_{\pi} , \xi_{\pi} \right>| = |\left< \pi'(s)V \xi_{\pi}, V \xi_{\pi}\right>| < \|s\|.
$$

 This is a contradiction. Hence $\pi' \sim \pi$. Therefore, $\pi' = U^*\pi U$ for some unitary $U:H_{\pi'} \rightarrow H_{\pi}$. Hence $\Phi= V^*\pi'V $ = $V^*U^*\pi UV = V_1^*\pi V_1$ where $V_1 = UV$ is an isometry. Thus, $\Phi(s) = V_1^*\pi(s)V_1$ for every $s \in S$. Since $\Phi_{\mid_{S}} = \pi_{\mid_{S}}$, we have $ \pi(s) = V_1^*\pi(s)V_1$ for every $s\in S$. By our assumption $\pi$ is weak boundary representation, hence $\pi(a) =  V_1^*\pi(a)V_1$ for all $a \in A$ and therefore $\pi(a) = \Phi(a)$ for all $a \in A$. Thus $\pi= \Phi$.
\end{proof}

Following examples illustrates the above theorem.

\begin{example}
Let the Volterra integration operator $V$ acting on the Hilbert space $H=L^2[0,1]$ given by
$$Vf(x)=\int_0^x f(t)dt,~~~~~~~ f\in L^2[0,1] .$$
It is well known that $V$ generates the $C^*$-algebra $K=K(H)$ of all compact operators. Let $S=span\left(V,V^*,V^2,V^{2*}\right)$ and $S$ is hyperrigid (\cite{Arv2}, Theorem 1.7). Let $\tilde{S}=S+\mathbb{C}\cdot \mathbf{1}$ be an operator system generating the $C^*$-algebra $\tilde{A}=K+\mathbb{C}\cdot \mathbf{1}$. The irreducible representations of $\tilde{A}$ are $\pi$ and $\rho$ given by
$$\pi(T+\lambda\mathbf{1})=T,~~~\text{for}~~~ T\in K,~ \lambda \in \mathbb{C}  $$
$$\rho(T+\lambda\mathbf{1})= \lambda, ~~~\text{for}~~~ T\in K,~ \lambda \in \mathbb{C} $$
Infact these are the only two irreducible representations upto unitary equivalence. $\tilde{S}$ is a hyperrigid operator system (\cite{Arv2}, Theorem 2.1) implying that $\pi$ and $\rho$ are boundary representations for $\tilde{S}$ of $\tilde{A}$. Also $\tilde{S}$ is quasi hyperrigid and therefore $\pi$, $\rho$ are weak boundary representations for $\tilde{S}$. 

Let $V+V^*\in \tilde{S}$ be the projection on the space of constants and let the constant function $1\in L^2[0,1]$, $||1||=1$

$$|\left< \pi(V+V^*)1,1 \right> |=1=||V+V^*||. $$ 
For all $\xi_\rho\in \mathbb{C}$, $||\xi_\rho||=1$.
$$|\left< \rho(V+V^*)\xi_\rho,\xi_\rho \right> |=|\left< 0\xi_\rho,\xi_\rho \right> |=0<||V+V^*||. $$ 
Therefore $\pi$ is a weak peak point.

Let $\mathbf{1}\in \tilde{S}$ and $1\in \mathbb{C}$, $||1||=1$
$$|\left< \rho(\mathbf{1})1,1 \right> |=1=||\mathbf{1}||. $$
For all $\xi_\pi\in L^2[0,1]$, $||\xi_\pi||=1$
$$|\left< \pi(\mathbf{1})\xi_\pi,\xi_\pi \right> |=|\left< 0\xi_\pi,\xi_\pi \right> |=0<||\mathbf{1}||. $$
Hence $\rho$ is a weak peak point. Also $\pi$ and $\rho$ restricted to $\tilde{S}$ are pure.

\end{example}

\begin{example}
Let $G = span\left(I,S,S^*,SS^{*}\right)$, where $S$ is the unilateral right shift in $B(H)$ and $I$ the identity operator. Let $A=C^{*}(G)$ be the $C^*$-algebra generated by $G$. We have,  $K(H) \subseteq A$. $A/K(H)\cong {C}(\mathbb{T})$ is commutative, where $\mathbb{T}$ denotes the unit circle in $\mathbb{C}$ and the spectrum $\hat{A}$ of $A$ can be identified with $\{Id\}\cup \mathbb{T}$. Since $S$ is an isometry, $G$ is hyperrigid (\cite{Arv2}, Theorem 3.3) and this will imply that all the irreducible representations of $A$ are boundary representations for $S$. Clearly $G$ is quasi hyperrigid, so all the irreducible representations are weak boundary representations for $S$.
 
Now we prove that identity representation $Id$ of $A$ is a weak peak point for $G$. Let $e_1=(1,0,0...,0)$  and let $E=I-SS^*\in G$ be the rank one projection. We have $|\left< Id(E)e_1, e_1 \right> |= 1=||E||$ and for any irreducible representation $\pi$ which is not equivalent to identity, $\pi(E)=0$. So we have $|\left< \pi(E)\eta,\eta \right> |=0<||E||$ for all unit vectors $\eta \in H_{\pi}$. This proves that $Id$ is a weak peak point. Also $Id_{|_G}$ is pure. 
\end{example}

Now we give a `lighter' version of weak peak points where we don't insist on the condition (ii) being true for all unit vectors of the corresponding Hilbert space.
\begin{definition}
Let $A$ be a unital $C^*$-algebra and $S$ be an operator system of $A$ such that $A = C^*(S)$, the $C^*$-algebra generated by $S$. An element $\pi$ of $\hat{A}$ is called a {\it quasi weak peak point} for $S$ if there exists $s \in S$ such that
\begin{itemize}
\item [(i)] $|\langle \pi(s) \xi_{\pi} , \xi_{\pi}  \rangle| =\|s\|$ for some  $\xi _\pi \in H_\pi$ with $\|\xi_{\pi}\| = 1$,
\item [(ii)] $|\langle \sigma(s) \xi_{\sigma}  , \xi_{\sigma}  \rangle|< \|s\|$ for some  $\xi _\sigma \in H_\sigma$ with $\|\xi_{\sigma}\| = 1$,
\end{itemize}where $\sigma$ is any  irreducible representation not equivalent to $\pi$.
\end{definition}

We now give a few examples.

\begin{example}
 For each $\lambda \in \mathbb C$, let $T_{\lambda} \in M_{3}(\mathbb C)$ be given by  $T_{\lambda}= \left[ \begin{array}{ccc} 0&1&0 \\ 0&0&0\\0&0&\lambda\end{array} \right]$. Let $S_{T_{\lambda}} = span \{I,{T_{\lambda}},{{T_{\lambda}}}^*\}$ denote the operator system generated by $T_{\lambda}$. Now, let $A = C^*(S_{T_{\lambda}})=M_2(\mathbb{C})\oplus \mathbb{C}$ be the $C^{*}$-algebra generated by $S_{T_{\lambda}}$. Consider the map $\pi : A \rightarrow \mathbb C$ which sends each $X \in A$ to its $(3,3)-$ entry. Thus, $\pi$ is an irreducible representation of $A$ onto $\mathbb C$. Define another irreducible representation $\rho : A \rightarrow M_{2}(\mathbb C)$ by $\rho (X) = V^*XV$, where $V = \left[ \begin{array}{cc} 1&0 \\ 0&1\\0&0\end{array} \right] $. It can be proved that $\rho$ and $\pi$ are the only irreducible representations (up to unitary equivalence) of $A$. We will prove that $\pi$ is quasi weak peak point for $\lambda =\frac{1}{2}$. Let $S=\left[ \begin{array}{ccc} 0&1&0 \\ 1&0&0\\0&0&1\end{array} \right]$ and $\xi_{\pi}=1$, $|\langle \pi(S)\xi_\pi,\xi_\pi  \rangle|=1=\|S\|$. Let $\xi_\rho=\left[ \begin{array}{c} 1 \\ 0\end{array} \right]$, $|\langle \rho(S)\xi_\rho,\xi_\rho  \rangle|=\left|\left\langle\left[ \begin{array}{cc} 0&1 \\ 1&0\end{array} \right]   \left[ \begin{array}{c} 1 \\ 0\end{array} \right],\left[ \begin{array}{c} 1 \\ 0\end{array} \right]\right\rangle\right|=0< \|S\|$. Hence $\pi$ is a quasi weak peak point for $\lambda =\frac{1}{2}$.  
\end{example}
                 
\begin{example}(\cite{Hop}, Page 488)
Let $X = [0,1]$ and $A = C(X)$. Consider $f: [0,1] \rightarrow \mathbb{R}$ which is a strictly positive and strictly increasing continuous function. Consider the $C^{\ast}$-algebra $A \otimes M_2$. Let $G$ be operator system in  $A\otimes M_2$ spanned by
$ I = \left[ \begin{array}{cc} 1&0 \\ 0&1 \end{array} \right]$ and $F = \left[ \begin{array}{cc} 0&0 \\ f&0 \end{array} \right]$.
Here $C^{\ast}(G)  = A\otimes M_2$, and the irreducible representations of $A \otimes M_2$ on $\mathbb{C}^2$ are given by $\rho_t , t \in [0,1]$ where $\rho_t(F)=\left[ \begin{array}{cc} 0&0 \\ f(t)&0 \end{array} \right]$ represents the point evaluation at $t$ and by \cite{Hop}, the only boundary representation for $G$ in $A \otimes M_2$ is $\rho_1$. We will show that $\rho_1$ is a weak peak point for $G$. Let $S=\left[ \begin{array}{cc} 0&f \\ f &0 \end{array} \right]$ and $\xi_{\rho_1}=\left[ \begin{array}{c}\frac{1}{\sqrt{2}} \\ \frac{1}{\sqrt{2}} \end{array} \right]$ for $t=1$, 
$\left|\left\langle  \rho_1\left[ \begin{array}{cc} 0&f \\ f &0 \end{array} \right]\left[ \begin{array}{c}\frac{1}{\sqrt{2}} \\ \frac{1}{\sqrt{2}} \end{array} \right],\left[ \begin{array}{c}\frac{1}{\sqrt{2}} \\ \frac{1}{\sqrt{2}} \end{array} \right]  \right\rangle\right|=\left|\left\langle \left[ \begin{array}{cc} 0&f(1) \\ f(1) &0 \end{array} \right]\left[ \begin{array}{c}\frac{1}{\sqrt{2}} \\ \frac{1}{\sqrt{2}} \end{array} \right],\left[ \begin{array}{c}\frac{1}{\sqrt{2}} \\ \frac{1}{\sqrt{2}} \end{array} \right]  \right\rangle\right|=|f(1)|=||S||.  $
And  for all $t\in [0,1)$, $|\langle \rho_t(S)\xi_t,\xi_t  \rangle| < |f(1)|$ for all $\xi_t \in H_{\rho_t}$. Hence $\rho_1$ is a weak peak point. 
\end{example}

\begin{remark}
In the classical case, when $A= C(X)$, where $X$ is a compact Hausdorff space, irreducible representations correspond to point evaluation functionals and thereby precisely to the points of $X$. Let  $\pi_x$ be  the irreducible representation corresponding to $x \in X$. An $x_0 \in X$ is a weak peak point for $G \subseteq C(X)$ if there exists $g_0 \in G$ such that
$\left|\left<\pi_{x_0} (g_0) \xi_{\pi_{x_0}}, \xi_{\pi_{x_0}} \right>\right| = \|g_0\|$ for some $\xi_{\pi_{x_0}}\in H_{\pi_{x_0}}$ with $\|\xi_{\pi_{x_0}}\|=1$ and  $|\left<\pi_{x} (g_0) \xi_{\pi_{x}}, \xi_{\pi_{x}} \right>| <  \|g_0\|$  for all $\xi_{\pi_{x}}\in H_{\pi_{x}}$ with $\|\xi_{\pi_{x}}\|=1$, where $\pi_{x}$ is any  irreducible representation not equivalent to $\pi_{x_0}$. i.e., $g_0(x_0) = \|g_0\| $ and $ | g_0(x) |  < \|g_0\|$ for every $x \neq x_0$ which implies that $x_0$ is a peak point for $G$.  Hence in the classical case both weak peak points and peak points coincide. In the classical case we can prove that quasi weak peak points and peak points also coincide using similar arguments. Hence all the three notions viz. weak peak points, quasi weak peak points and peak points coincide in the classical case.
\end{remark}

\begin{remark}
It is clear that the concepts and the corresponding analysis is more based on a modest setting than the much stronger notions employed by Arveson in his series of articles. However, it is revealed that there are non-trivial questions related to the structure of certain interesting operator spaces associated with isometries. In (\cite{BL}, Section 1.5) there are two more notions of classical peak points depending on the maps of which the non commutative analogue is yet to be studied.  
\end{remark}

{\bf{Acknowledgment}:}
 A part of this work was done during the `Discussion meeting on Functional Analysis' held during 22nd May - 25th May, 2015 and 14th August - 16th August, 2015 at the Kerala School of Mathematics(KSoM), Kozhikode, Kerala, India. The authors gratefully acknowledge the financial assistance provided by KSoM under its `Research Workshop' programme.

The first author is thankful to UGC for providing Emeritus Fellowship and the support extended to complete this work. The third author is thankful to CSIR for the fellowship provided for carrying out his research.

\bibliographystyle{amsplain}

\end{document}